\documentclass[a4paper,11pt]{article}
\usepackage[T1]{fontenc}
\usepackage[utf8]{inputenc}
\usepackage{lmodern,microtype}
\usepackage{amsmath,amssymb,amsthm,mathtools}
\usepackage{tikz-cd}
\usepackage[margin=27mm]{geometry}
\usepackage{enumitem,booktabs,array}
\usepackage[hidelinks]{hyperref}
\hypersetup{pdftitle={A Survey on Lawvere's Fixed-Point Theorem},pdfauthor={Joaquim Reizi Higuchi}}
\newtheorem{theorem}{Theorem}[section]
\newtheorem{proposition}[theorem]{Proposition}
\newtheorem{lemma}[theorem]{Lemma}
\newtheorem{corollary}[theorem]{Corollary}
\theoremstyle{definition}
\newtheorem{definition}[theorem]{Definition}
\newtheorem{example}[theorem]{Example}
\theoremstyle{remark}
\newtheorem{remark}[theorem]{Remark}
\newcommand{\Set}{\mathbf{Set}}
\newcommand{\id}{\mathrm{id}}
\newcommand{\ev}{\mathrm{ev}}
\newcommand{\Hom}{\mathrm{Hom}}
\newcommand{\pair}[2]{\langle #1,#2\rangle}

\title{A Survey on Lawvere's Fixed-Point Theorem}
\author{Joaquim Reizi Higuchi}
\date{October 12, 2026\\Second edition, expanded survey}

\begin{document}
\maketitle
\begin{abstract}
Lawvere's fixed-point theorem organizes self-reference around a representation
hypothesis and a diagonal calculation. This survey develops its set-theoretic,
finite-product, and cartesian closed formulations while distinguishing
representation on points from equality of whole morphisms. Detailed
applications include Cantor's theorem, arithmetic diagonalization, a truth-schema
obstruction, an incompleteness consequence with explicit assumptions, the
halting problem, and Kleene's recursion theorem. A powerset dcpo supplies a
nontrivial positive representation of continuous maps and illustrates the
difference between fixed-point existence, a uniform operator, and a least
fixed point. Further sections compare parameterized recursion and trace,
substructural and relational generalizations, constructive witness data,
and topological variants. Throughout, the relevant maps, equality relations,
and additional hypotheses are kept explicit. Worked examples use finite tables, Boolean equations, numerical function codes,
factorial approximations, and self-output programs. The survey is intended for
readers beginning category theory who wish to understand both the common
diagonal mechanism and the work required in its applications.
\end{abstract}
\tableofcontents

\section{Introduction}\label{sec:introduction}
A diagonal argument compares an object with its own description. In a
table of functions, it examines the entry where the row index and the
argument coincide. Lawvere expressed this pattern categorically
\cite{Lawvere1969}; Yanofsky explains its use in several logical and
computational settings \cite{Yanofsky2003}. The central implication is
conditional. A category being cartesian closed does not, by itself,
give fixed points for all endomorphisms.

Here is the pattern in sets. Suppose a function
$E:A\times A\to B$ lists every function $A\to B$ among its rows
$x\mapsto E(a,x)$. For an arbitrary $g:B\to B$, consider the
diagonal function $h(x)=g(E(x,x))$. A row representing $h$, evaluated
at its own index $a$, gives
\[
 E(a,a)=h(a)=g(E(a,a)).
\]
Thus $E(a,a)$ is a fixed point. If $g$ has no fixed point, the same
calculation shows that such a representation table cannot exist.

The burden of an application is therefore to identify its representation
table and justify the representation hypothesis. This survey does not
claim a new fixed-point theorem. Its aim is to make the standard
argument, its categorical notation, and the limits of its applications
explicit. Section~\ref{sec:sets} introduces the set-theoretic argument;
Sections~\ref{sec:preliminaries}--\ref{sec:ccc} give its categorical form.
After the diagonal consequences in Section~\ref{sec:applications},
Sections~\ref{sec:logic} and~\ref{sec:recursion} spell out the
encoding and partiality needed in logic and computability.
Sections~\ref{sec:uniform} and~\ref{sec:models} construct a uniform
operator under a split representation, give a positive model, and
compare stronger notions of solution. Sections~\ref{sec:generalizations}
and~\ref{sec:constructive} discuss weakened structure and constructive
data; Section~\ref{sec:literature} treats the scope of topological
and related applications. The closing comparison and directions
identify the assumptions needed when transferring the argument.
The appendix records universal-property checks used in the proofs.

The historical path includes Lawvere's 1969 paper and its later
reprint, Soto-Andrade--Varela's extension and converse questions
\cite{SotoVarela1984}, Yanofsky's exposition \cite{Yanofsky2003},
and Roberts's structural generalization \cite{Roberts2023}.
The literature discussed here extends through October 2026.
These sources concern related but different statements; the aim
is to compare them rather than assign a single formulation to all.

The worked examples are chosen to make the hypotheses testable.
Examples~\ref{ex:diagonal-table} and~\ref{ex:xor} show why a row can
be missing; Example~\ref{ex:gsets} distinguishes equality on global
points from equality of maps. Numerical application codes, finite
recursive equations, and factorial approximations connect representation
to computation. The self-output and diagonal-language examples separate
positive recursion from limitations on recognition. These illustrations
specialize the cited frameworks; they are not claims of new theorems.

\section{The argument in sets}\label{sec:sets}
\begin{definition}[A representation table]
For sets $A,B$, a function $E:A\times A\to B$ is
\emph{row-surjective} if
\[
 \forall h:A\to B\quad \exists a\in A\quad
 \forall x\in A\quad E(a,x)=h(x).
\]
The index $a$ may depend on $h$. Its uniqueness is not required.
\end{definition}

\begin{theorem}[The set-theoretic diagonal argument]\label{thm:set}
If $E:A\times A\to B$ is row-surjective, every function
$g:B\to B$ has a fixed point.
\end{theorem}
\begin{proof}
Define $d(x)=E(x,x)$ and $h=g\circ d$. Choose a row index $a$ for
$h$. Substituting $x=a$ in $E(a,x)=h(x)$ gives
$d(a)=h(a)=g(d(a))$. Hence $b=d(a)$ satisfies $g(b)=b$.
\end{proof}

\begin{example}[Finding the missing row]\label{ex:diagonal-table}
Let $A=\{0,1,2\}$, $B=\{0,1\}$, and consider this table:
\[
\begin{array}{c|ccc}
 E(a,x)&x=0&x=1&x=2\\\hline
 a=0&0&0&1\\
 a=1&1&0&0\\
 a=2&0&1&1
\end{array}
\qquad
\begin{array}{c|ccc}
 x&0&1&2\\\hline
 d(x)=E(x,x)&0&0&1\\
 h(x)=1-d(x)&1&1&0
\end{array}
\]
The function $h$ is missing from the rows: row $0$ differs from $h$
at $x=0$, row $1$ at $x=1$, and row $2$ at $x=2$.
The three comparisons are instances of
$E(a,a)\ne1-E(a,a)$. A larger table cannot repair this particular
mechanism, because each new row is tested at its own diagonal entry.
This calculation illustrates the general diagonal argument, rather
than replacing it with a check of one finite table
\cite{Lawvere1969,Yanofsky2003}.
\end{example}

\begin{example}[Existence need not mean uniqueness]\label{ex:nonunique}
Let $A=\{0,1\}$ and $B=\{*\}$. The unique table $E$ is
row-surjective, since there is only one function $A\to B$.
Both $0$ and $1$ represent that function. There is no unique row
index, even though the fixed point in $B$ is unique.
\end{example}

\begin{proposition}[The set-theoretic boundary]\label{prop:set-boundary}
For a function $E:A\times A\to B$, row-surjectivity holds if and
only if $A$ is nonempty and $B$ has exactly one element.
\end{proposition}
\begin{proof}
If $A$ is empty, the unique function $A\to B$ has no representing
index. If $A$ is nonempty and $E$ exists, then $B$ is nonempty,
since $E(a,a)\in B$ for any $a\in A$.
If $B$ contains distinct elements $b_0,b_1$, define
\[
 g(b)=\begin{cases}b_1,&b=b_0,\\b_0,&b\ne b_0.\end{cases}
\]
This function has no fixed point, contrary to
Theorem~\ref{thm:set}. Thus $B$ must be a singleton.
Conversely, if $A$ is nonempty and $B$ is a singleton, every row
represents the unique function $A\to B$.
\end{proof}

This boundary is a feature of full function spaces in $\Set$.
It does not make the categorical theorem vacuous in every category.
In a different category, only its morphisms are required to be
represented, and their available global points may differ from
ordinary set elements. A nontrivial application must specify that
category and its representation hypothesis.

\section{Categorical preliminaries}\label{sec:preliminaries}
\subsection{Objects, morphisms, and global points}
A category has objects, morphisms between objects, associative
composition, and identity morphisms. We write $v\circ u$ for
``first $u$, then $v$''. If $u:X\to Y$ and $v:Y\to Z$, this
composite has type $X\to Z$.

\begin{definition}[Terminal object and global point]
A terminal object $1$ has a unique morphism $!_X:X\to1$ for
every object $X$. A \emph{global point} of $A$ is a morphism
$a:1\to A$.
\end{definition}
In $\Set$, $1$ can be any singleton and a global point is the
same as an element. In an arbitrary category, global points need
not detect all distinctions between morphisms. We will not assume
that they do.

\subsection{Products and the diagonal}
\begin{definition}[Binary product]
A product $A\times B$ has projections $\pi_A:A\times B\to A$
and $\pi_B:A\times B\to B$ such that, for every
$u:X\to A$ and $v:X\to B$, there is a unique
$\pair{u}{v}:X\to A\times B$ satisfying
\[
 \pi_A\circ\pair{u}{v}=u,\qquad
 \pi_B\circ\pair{u}{v}=v.
\]
A category has \emph{finite products} when it has a terminal
object and binary products.
\end{definition}
For $p:A\to A'$ and $q:B\to B'$, define
$p\times q=\pair{p\circ\pi_A}{q\circ\pi_B}$, a morphism
$A\times B\to A'\times B'$. The \emph{diagonal} is
\[
 \Delta_A=\pair{\id_A}{\id_A}:A\to A\times A.
\]
It duplicates an input; it is not an identity morphism on
$A\times A$. For a global point $a:1\to A$, define its row
insertion map by
\[
 i_a=\pair{a\circ!_A}{\id_A}:A\to A\times A.
\]
In sets, these maps send $x$ to $(x,x)$ and to $(a,x)$,
respectively. We use pairing to avoid suppressing the unit
isomorphisms $1\times A\cong A$ and $A\times1\cong A$.

\begin{lemma}[Pairing and precomposition]\label{lem:pairing}
For $u:X\to A$, $v:X\to B$, and $k:Y\to X$,
\[
 \pair{u}{v}\circ k=\pair{u\circ k}{v\circ k}.
\]
In particular, $\Delta_A\circ a=\pair{a}{a}$ and
$i_a\circ a=\pair{a}{a}$.
\end{lemma}
\begin{proof}
Both sides of the first equality have the same two composites
with the product projections, so product uniqueness makes them
equal. For the last equality, $!_A\circ a=\id_1$, since there
is only one morphism $1\to1$.
\end{proof}

\subsection{Exponentials and currying}
\begin{definition}[Exponential object]
An exponential of $B$ by $A$ is a pair $(B^A,\ev)$, where
$\ev:B^A\times A\to B$, with the following property:
for every $u:X\times A\to B$, there is a unique morphism
$\Lambda u:X\to B^A$ such that
\[
 \ev\circ((\Lambda u)\times\id_A)=u.
\]
A category with finite products and exponentials for all $A,B$
is \emph{cartesian closed} (a CCC).
\end{definition}
For sets, $B^A$ is the set of functions $A\to B$,
$\ev(\varphi,x)=\varphi(x)$, and
$(\Lambda u)(y)(x)=u(y,x)$. The evaluation morphism is part of
the exponential structure. Its uniqueness should not be asserted
after fixing only the underlying object; see
Appendix~\ref{app:exponential}.

For $h:A\to B$, its name is the global point
\[
 \widehat h=\Lambda(h\circ\pi_A):1\to B^A,
 \qquad \pi_A:1\times A\to A.
\]
The defining equality for currying, precomposed with
$\pair{\id_1}{x}:1\to1\times A$, gives
\begin{equation}\label{eq:name}
 \ev\circ\pair{\widehat h}{x}=h\circ x
 \quad\text{for every }x:1\to A.
\end{equation}
More generally, currying gives a bijection
$\Hom(1,B^A)\cong\Hom(A,B)$ after the unit isomorphism is
included. This is a bijection of hom-sets, not an assumption
that all objects behave like sets of global points.

\begin{example}[Currying XOR does not give a surjection]\label{ex:xor}
For $A=B=\{0,1\}$, write a function $A\to B$ as its two values
at $0,1$. Thus $B^A=\{00,01,10,11\}$. Let
$u(a,x)=a\mathbin{\oplus}x$, where $\oplus$ is exclusive-or.
Its transpose satisfies
\[
 (\Lambda u)(0)=01,\qquad (\Lambda u)(1)=10.
\]
Evaluation reads the appropriate coordinate: for example,
$\ev(01,1)=1$ and $\ev(10,1)=0$. The transpose is not
surjective, since neither constant function $00$ nor $11$ is
in its image. For negation $g(b)=1-b$, the diagonal of $u$ is
constantly $0$ and $g$ composed with that diagonal is constantly
$1$. Its name is $11$, precisely a missing row.
The existence of an exponential and of currying is therefore
not the representation hypothesis of Lawvere's theorem.
\end{example}

\section{Representation and Lawvere's theorem}\label{sec:lawvere}
\subsection{Two precise representation conditions}
Let $\mathcal C$ have finite products and let
$E:A\times A\to B$.
\begin{definition}[Strong and weak row representation]\label{def:rows}
The morphism $E$ is \emph{row-surjective} if every $h:A\to B$
has a global point $a:1\to A$ such that
\begin{equation}\label{eq:strong-row}
 E\circ i_a=h.
\end{equation}
It is \emph{weakly row-surjective} if every $h:A\to B$
has a global point $a:1\to A$ such that
\begin{equation}\label{eq:weak-row}
 \forall x:1\to A,\qquad
 E\circ\pair{a}{x}=h\circ x.
\end{equation}
No uniqueness of $a$ is required in either condition.
\end{definition}
The quantifier order matters: one point $a$ must work for every
$x$, rather than a different representing point for each input.
Equation~\eqref{eq:strong-row} implies
\eqref{eq:weak-row} by precomposition with $x$. The converse
needs a point-detection hypothesis: if two morphisms from $A$
are equal whenever all their composites with global points
are equal, the weak condition implies the strong condition.
Without that hypothesis, we do not identify them.
In $\Set$, both conditions reduce to the definition in
Section~\ref{sec:sets}.

\begin{example}[Global points need not distinguish maps]\label{ex:gsets}
Consider sets with an action of the two-element group $C_2$, and
maps preserving the action. Let $A=\{*,a,b\}$, with the nonidentity
group element fixing $*$ and exchanging $a,b$. Give
$B=\{0,1\}$ the trivial action. The following maps are equivariant:
\[
\begin{array}{c|ccc}
 x&*&a&b\\\hline
 u(x)&0&0&0\\
 v(x)&0&1&1
\end{array}
\]
The terminal object is a singleton with trivial action. A map from
it to $A$ must select a group-fixed element; the only such element
is $*$. Thus $u\circ x=v\circ x$ for every global point
$x:1\to A$, although $u\ne v$.
This is a counterexample to point detection, not an example of a
weakly row-surjective table. It explains why a pointwise equality
cannot automatically be upgraded to equality of morphisms.
\end{example}

\begin{theorem}[Lawvere's diagonal fixed-point theorem]\label{thm:lawvere}
Let $\mathcal C$ be a category with finite products. If
$E:A\times A\to B$ is weakly row-surjective, then every
endomorphism $g:B\to B$ has a global fixed point: there is
$b:1\to B$ such that $g\circ b=b$.
\end{theorem}
\begin{proof}
Form the diagonal evaluation and its composite with $g$:
\[
 d=E\circ\Delta_A:A\to B,\qquad
 h=g\circ d:A\to B.
\]
Choose a representing point $a:1\to A$ for $h$ in the weak
row condition. Set $b=d\circ a:1\to B$. Use that condition
at the particular global point $x=a$. By
Lemma~\ref{lem:pairing},
\begin{align*}
 b&=d\circ a
   =E\circ\Delta_A\circ a
   =E\circ\pair{a}{a}\\
  &=h\circ a
   =g\circ d\circ a
   =g\circ b.
\end{align*}
All morphisms in this chain have domain $1$ and codomain $B$.
\end{proof}

The construction uses $\Delta_A$ to turn a two-variable table
into a one-variable map, and then uses $a$ to turn that map
into a global point. The following diagrams express the two
stages. The first commutes by the definition of $h$; the last
commutes by the calculation in the proof.
\[
\begin{tikzcd}[column sep=large]
 A \arrow[r,"\Delta_A"] \arrow[d,"h"']
 & A\times A \arrow[r,"E"] & B \arrow[d,"g"]\\
 B \arrow[rr,equal] && B
\end{tikzcd}
\qquad
\begin{tikzcd}
 1 \arrow[r,"b"] \arrow[dr,"b"'] & B \arrow[d,"g"]\\
 & B
\end{tikzcd}
\]

\begin{corollary}[The obstruction form]\label{cor:obstruction}
If $B$ has an endomorphism with no global fixed point, then
there is no weakly row-surjective $E:A\times A\to B$ for
any object $A$.
\end{corollary}
\begin{proof}
Such an $E$, together with that endomorphism, would contradict
Theorem~\ref{thm:lawvere}.
\end{proof}

\section{The cartesian closed formulation}\label{sec:ccc}
Let $\mathcal C$ be cartesian closed, and let
$f:A\to B^A$. Its associated table is
\[
 E_f=\ev\circ(f\times\id_A):A\times A\to B.
\]
\begin{definition}[Point-surjectivity in the exponential form]\label{def:point-surj}
We call $f$ \emph{point-surjective} if
\begin{equation}\label{eq:point-surj}
 \forall h:A\to B\quad \exists a:1\to A\quad
 f\circ a=\widehat h.
\end{equation}
We call $f$ \emph{weakly point-surjective} if its table $E_f$
is weakly row-surjective in Definition~\ref{def:rows}.
\end{definition}
We keep the adjective ``weakly'' explicit. The equation in
\eqref{eq:point-surj} is equality of two morphisms
$1\to B^A$, whereas the weak condition asks for equality
only after evaluation at all global points of $A$.

\begin{lemma}\label{lem:point-row}
The morphism $f$ is point-surjective if and only if $E_f$ is
row-surjective. Consequently, point-surjectivity implies
weak point-surjectivity.
\end{lemma}
\begin{proof}
For any $a:1\to A$,
\[
 E_f\circ i_a
 =\ev\circ\pair{f\circ a\circ!_A}{\id_A}.
\]
Under the currying bijection
$\Hom(A,B)\cong\Hom(1,B^A)$, the name of this row is
$f\circ a$. Thus the row equals $h$ exactly when
$f\circ a=\widehat h$. The implication to the weak condition
then follows from Definition~\ref{def:rows}.
\end{proof}

\begin{corollary}[Lawvere's theorem in a CCC]\label{cor:ccc}
If $f:A\to B^A$ is weakly point-surjective, every
endomorphism $g:B\to B$ has a global fixed point.
In particular, this conclusion holds under the
point-surjectivity condition~\eqref{eq:point-surj}.
\end{corollary}
\begin{proof}
Apply Theorem~\ref{thm:lawvere} to $E_f$.
\end{proof}
For the stronger condition, the construction can be recorded
in the following fully typed form:
\[
\begin{aligned}
 d&=\ev\circ(f\times\id_A)\circ\Delta_A:A\to B,\\
 h&=g\circ d:A\to B,\\
 f\circ a&=\widehat h:1\to B^A,\\
 b&=d\circ a
   =\ev\circ\pair{f\circ a}{a}:1\to B.
\end{aligned}
\]
The name identity~\eqref{eq:name} gives
$b=h\circ a=g\circ b$. Neither equality $h=b$ nor a unique
choice of $a$ is asserted.

\section{Diagonal consequences}\label{sec:applications}
\subsection{Cantor's theorem}
\begin{corollary}[Cantor]\label{cor:cantor}
For every set $A$, there is no surjection $A\to\mathcal P(A)$.
\end{corollary}
\begin{proof}
Suppose $q:A\to\mathcal P(A)$ were surjective. Let
$2=\{0,1\}$ and put $E(a,x)=1$ if $x\in q(a)$, and
$E(a,x)=0$ otherwise. Each function $h:A\to2$ is the
characteristic function of its set $\{x:h(x)=1\}$, so
surjectivity of $q$ would make $E$ row-surjective. But
$g:2\to2$, $g(t)=1-t$, has no fixed point. This contradicts
Theorem~\ref{thm:set}.
\end{proof}
The usual diagonal subset is
$D=\{x\in A:x\notin q(x)\}$. Its characteristic function
is $\chi_D(x)=1-E(x,x)$.
This example uses the obstruction form of the theorem, not
the construction of a nontrivial fixed point in $\Set$.

\subsection{A limitation on total computable enumerations}
\begin{proposition}\label{prop:computable}
There is no total computable function
$F:\mathbb N\times\mathbb N\to2$ whose rows include every
total computable function $\mathbb N\to2$.
\end{proposition}
\begin{proof}
If such an $F$ existed, $h(n)=1-F(n,n)$ would be a total
computable binary-valued function. A row $k$ representing $h$
would satisfy $F(k,k)=h(k)=1-F(k,k)$, impossible for a value
in $2$.
\end{proof}
This statement concerns a uniformly total computable
evaluation function. It does not rule out enumerating partial
computable functions by programs: their evaluation may be
undefined. Nor does it supply a proof of every incompleteness
or recursion theorem from categorical notation alone. Such
applications require their own encoding and representation
arguments; Yanofsky gives a broader account
\cite{Yanofsky2003}.

\section{Self-reference in logic}\label{sec:logic}
The table notation makes the diagonal pattern easy to recognize, but
logical applications have an additional layer: expressions must be
encoded, substitution must be represented, and the relevant equality
must be interpreted in a theory. Yanofsky's survey
\cite{Yanofsky2003} is an important guide to these connections.
We give a code-based construction so that the operations do not depend
on an unproved choice of representative of a logical equivalence class.

\subsection{From Russell's paradox to restricted representation}
Unrestricted comprehension would assert that, for every predicate
$P(x)$, there is a set $R$ with
$\forall x\,(x\in R\leftrightarrow P(x))$. Taking
$P(x)$ to be $x\notin x$ and then substituting $x=R$ gives
\[
 R\in R\quad\longleftrightarrow\quad R\notin R.
\]
This is inconsistent in classical logic. In ordinary set theory the
comprehension principle is restricted: separation constructs a subset
of an already given set, not a set containing every object satisfying
an arbitrary predicate. Thus the obstruction concerns the proposed
representation principle. It is not a contradiction in the set-theoretic
fixed-point theorem.

The comparison with Cantor is precise at the level of the diagonal
operation: membership is the two-variable relation, negation changes
the diagonal truth value, and a supposed representing object is tested
against itself. The available predicates and the range of variables
must nevertheless be specified before this becomes a categorical
application. A universal set equipped with unrestricted predicate
representation is not an object one may simply assume in $\Set$.

\begin{example}[Grelling's adjective example]\label{ex:grelling}
For a motivation discussed by Yanofsky \cite{Yanofsky2003}, write
$E(a,x)=1$ when adjective $a$ describes word $x$, and $0$
otherwise. ``Polysyllabic'' is itself a word of several syllables,
whereas ``monosyllabic'' is not a one-syllable word. Thus their
diagonal entries differ. The putative property
$h(x)=1-E(x,x)$ means ``does not describe itself''.
If an adjective $a_0$ in the same domain represents $h$, then
\[
 E(a_0,a_0)=h(a_0)=1-E(a_0,a_0),
\]
which has no Boolean solution. The obstruction is to combining
unrestricted descriptive closure with this stable two-valued
interpretation. Natural-language description is not being assumed
to supply an already well-defined total mathematical evaluator.
The example motivates the formal representation question; it does
not replace a formal semantics of natural language.
\end{example}

\subsection{Coding, substitution, and the diagonal lemma}
Fix the language $L$ of first-order arithmetic, a distinguished variable
$v$ of $L$, and a conventional effective coding of all formulas and
proofs in $L$. Write $\ulcorner\theta\urcorner$ for the natural-number
code of an expression $\theta$, and $\overline n$ for the numeral
naming the natural number $n$. These have different types: a code is a
number, while a numeral is a term in the language.

Let $T$ be a classical theory in this same language $L$, extending
Peano arithmetic and having a computably enumerable axiomatization. For definiteness use a proof
coding in which evidence for the enumeration of axioms is part of a
proof, so that proof checking is primitive recursive. The assumption
that $T$ extends Peano arithmetic is convenient rather than optimal.
There are weaker arithmetic bases in which the diagonal lemma holds;
we do not need that optimization here.

Let $\operatorname{Sub}(n,m)$ be a total primitive recursive
code-transformer which, when $n$ codes a formula whose only possible
free variable is $v$, gives the code of the result of substituting
$\overline m$ for $v$. Assign an arbitrary fixed default value to
ill-formed inputs. Put
\[
 d(n)=\operatorname{Sub}(n,n).
\]
Choose a graph formula $D(x,y)$ representing this function in arithmetic.
For each standard $n$, Peano arithmetic, and hence $T$, proves
\begin{equation}\label{eq:diag-graph}
 \forall y\bigl(D(\overline n,y)
       \leftrightarrow y=\overline{d(n)}\bigr).
\end{equation}
This use of representability is part of the application, not part of
Lawvere's abstract theorem.

\begin{theorem}[The arithmetic diagonal lemma]\label{thm:arith-diagonal}
For every arithmetic formula $\alpha(y)$ with no other free variables,
there is a sentence $G$ such that
\[
 T\vdash G\leftrightarrow\alpha(\overline{\ulcorner G\urcorner}).
\]
\end{theorem}
\begin{proof}
Define the one-variable formula
$\beta(v)=\exists y\,(D(v,y)\wedge\alpha(y))$ and let
$e=\ulcorner\beta\urcorner$. The only possible free variable of
$\beta$ is the distinguished variable $v$, so the substitution
operation defined above gives
$d(e)=\operatorname{Sub}(e,e)=\ulcorner\beta(\overline e)\urcorner$.
Take $G=\beta(\overline e)$ and $g=d(e)=\ulcorner G\urcorner$. Equation~\eqref{eq:diag-graph},
with $n=e$, gives
\[
 T\vdash
 G\leftrightarrow
 \exists y\,(y=\overline g\wedge\alpha(y))
 \leftrightarrow\alpha(\overline g).
\]
\end{proof}
No claim is made that $G$ and $\alpha(\overline{\ulcorner G\urcorner})$
are the same string. They are provably equivalent sentences. The
construction also does not require the map $\alpha$ to express
provability; that is a separate specialization.

\subsection{Why a quotient of formulas needs care}
For each fixed list of free variables, the Lindenbaum quotient
identifies formulas $\theta$ and $\eta$ when $T$ proves the universal
closure of $\theta\leftrightarrow\eta$. With no free variables this
is the quotient of sentences; below we use the quotient of one-variable
formulas. An operation defined using codes does not automatically
descend to this quotient. Suppose two
syntactically different one-variable formulas $\theta(v),\eta(v)$
are provably equivalent and have different codes $n,m$. Substitution
into $P(z)\equiv(z=\overline n)$ gives a true closed numeral equation
at $n$ and a false one at $m$. Therefore
\[
 [\theta]=[\eta]\quad\not\Rightarrow\quad
 [P(\overline{\ulcorner\theta\urcorner})]
     =[P(\overline{\ulcorner\eta\urcorner})]
\]
in a consistent arithmetic theory. The construction depends on the
syntactic representative, even though the formulas had the same
Lindenbaum class.

Vestrucci \cite{Vestrucci2026} develops a type-disciplined account
of diagonalization using primitive recursive code-transformers and
explicitly discusses this problem. A published correction concerns
notation for types and arrows \cite{VestrucciCorrection2026}.
The code construction above avoids assuming that coding respects
provable equivalence. It is a standard diagonal-lemma argument;
the cited work provides a contemporary comparison and further
computational detail, not the first discovery of the lemma.

\subsection{Tarski's obstruction to an internal truth predicate}
\begin{corollary}[No complete truth schema in a consistent theory]
A consistent $T$ as above cannot contain a formula $\operatorname{Tr}(y)$
for which
\[
 T\vdash\operatorname{Tr}(\overline{\ulcorner H\urcorner})
          \leftrightarrow H
 \qquad\text{for every sentence $H$ of its own language}.
\]
\end{corollary}
\begin{proof}
Apply Theorem~\ref{thm:arith-diagonal} to
$\alpha(y)=\neg\operatorname{Tr}(y)$. The resulting sentence $L$
satisfies
$T\vdash L\leftrightarrow\neg\operatorname{Tr}
(\overline{\ulcorner L\urcorner})$.
The assumed schema also gives
$T\vdash L\leftrightarrow\operatorname{Tr}
(\overline{\ulcorner L\urcorner})$.
Together these contradict classical consistency.
\end{proof}
This is a syntactic truth-schema formulation. Truth for a restricted
language, or a truth predicate in a stronger metalanguage, is not
excluded by this statement. Nor is a partial truth predicate required
to satisfy every instance of the displayed schema.

\subsection{A precise incompleteness consequence}
Let $\operatorname{Pr}_T(y)$ be the usual $\Sigma_1$ provability
formula for the chosen effective proof system. The diagonal lemma
supplies $G$ with
\begin{equation}\label{eq:godel}
 T\vdash G\leftrightarrow\neg\operatorname{Pr}_T
                  (\overline{\ulcorner G\urcorner}).
\end{equation}
\begin{proposition}[Two distinct assumptions]
If $T$ is consistent, $T\nvdash G$. If $T$ is also
$\Sigma_1$-sound, then $T\nvdash\neg G$.
Here $\Sigma_1$-sound means that every $\Sigma_1$ sentence proved
by $T$ is true in the standard natural numbers.
\end{proposition}
\begin{proof}
If $T\vdash G$, an actual finite proof of $G$ has a standard
code. Arithmetic verifies this proof and gives
$T\vdash\operatorname{Pr}_T(\overline{\ulcorner G\urcorner})$.
Equation~\eqref{eq:godel} also gives its negation, contradicting
consistency. Thus there is no proof of $G$.
If $T\vdash\neg G$, classical logic and~\eqref{eq:godel} give
$T\vdash\operatorname{Pr}_T(\overline{\ulcorner G\urcorner})$.
This is a $\Sigma_1$ sentence which is false in the standard
natural numbers, since there is no actual proof of $G$.
It contradicts $\Sigma_1$-soundness.
\end{proof}
This exposition uses a stronger hypothesis for the second conclusion
than is needed in some formulations of the first incompleteness
theorem. Rosser's improvement and the second incompleteness theorem
require further arguments not supplied here. In particular, the
second incompleteness theorem is not obtained merely by renaming the
endomorphism in Lawvere's basic statement.

The three logical examples separate the diagonal construction from
its interpretation. Russell and the truth schema meet negation under
an excessively strong representation assumption. The arithmetic
lemma instead produces a provable equivalence. Incompleteness then
uses the proof predicate and properties of the theory. These extra
steps are the mathematical substance of the application.

\section{Computability, partiality, and recursion}\label{sec:recursion}
A total evaluator for all total binary-valued computable functions is
impossible by Proposition~\ref{prop:computable}. Nevertheless, there
is a universal \emph{partial} computable function. There are also
fixed-point theorems for program indices. These statements are
compatible because their hypotheses and their equality relations
are different. Terwijn \cite{Terwijn2024} surveys this distinction and
its extensions to partial combinatory algebra and numberings.

\subsection{Universal evaluation and the halting obstruction}
Fix an acceptable numbering $(\varphi_e)_{e\in\mathbb N}$ of the
unary partial computable functions. Acceptability here includes a
universal partial computable evaluator $U(e,x)=\varphi_e(x)$ and
the parameterization property supplied by the $s$-$m$-$n$ theorem.
Write $\varphi_e(x)\downarrow$ when the computation halts.
Equality of partial functions means agreement of domains and of
values on that domain.

\begin{proposition}[The halting diagonal]
There is no total computable binary-valued function $H$ with
$H(e,x)=1$ exactly when $\varphi_e(x)\downarrow$.
\end{proposition}
\begin{proof}
If $H$ existed, the following procedure would be partial computable:
on input $n$, halt with output $0$ when $H(n,n)=0$, and otherwise
loop indefinitely. Let $p$ be its program index. If $H(p,p)=0$,
the procedure halts, contrary to that answer. If $H(p,p)=1$,
it does not halt, again contrary to the answer.
\end{proof}
The contradiction is not caused by a universal partial evaluator.
It is caused by the assumed total decision procedure for the
convergence of that evaluator. The two possible behaviors, halting
and divergence, must remain distinct in the model.

\begin{example}[Can a program always return zero?]\label{ex:rice-zero}
Suppose a total computable procedure decided whether a program
computes the total constant-zero function. Given $e,x$, compile
$P_{e,x}$ as follows: on any input $z$, simulate $\varphi_e(x)$;
if that simulation halts, output $0$; otherwise continue waiting.
The compiler is total although its produced program may diverge.
\[
 \varphi_e(x)\downarrow
 \quad\Longleftrightarrow\quad
 P_{e,x}\text{ computes the total constant-zero function}.
\]
If the simulation never halts, $P_{e,x}$ is undefined on every
input. The supposed decider would therefore decide the halting
problem, a contradiction. This is a concrete semantic property
of programs, of the kind covered by Rice's theorem; Online
source~\ref{online:waterloo} discusses that theorem.
Syntactic properties, such as whether a finite source text contains
a specified character, are not subject to this argument.
\end{example}

\begin{example}[A language outside the machine list]\label{ex:diagonal-language}
List all Turing machines $T_0,T_1,\ldots$ over $\{0,1\}$ and
list all finite binary words $w_0,w_1,\ldots$ effectively, without
repetition. Define
\[
 L_{\mathrm{diag}}=\{w_i:T_i\text{ does not accept }w_i\}.
\]
If $T_k$ recognized this language, then
\[
 T_k\text{ accepts }w_k
 \ \Longleftrightarrow\ w_k\in L_{\mathrm{diag}}
 \ \Longleftrightarrow\ T_k\text{ does not accept }w_k,
\]
which is impossible. Hence $L_{\mathrm{diag}}$ is not computably
enumerable. ``Does not accept'' includes both rejecting and
running forever. Defining the set in this way does not give an
algorithm for membership. This version appears in
Yanofsky \cite{Yanofsky2003}; it complements the halting example
by distinguishing recognition from decision and acceptance from
termination.
\end{example}

\subsection{Kleene's recursion theorem}
\begin{theorem}[Kleene, in an acceptable numbering]\label{thm:kleene}
For every total computable $f:\mathbb N\to\mathbb N$, there is
$e$ such that $\varphi_e=\varphi_{f(e)}$.
\end{theorem}
\begin{proof}
By the $s$-$m$-$n$ theorem there is a total computable function
$q$ such that the program with index $q(x)$ operates on input $z$
as follows: first compute $\varphi_x(x)$; if it halts with output $v$,
compute $\varphi_v(z)$; otherwise diverge. Consequently,
whenever $\varphi_x(x)\downarrow=v$,
\[
 \varphi_{q(x)}=\varphi_v.
\]
Choose an index $b$ for the total computable function $f\circ q$.
Then $\varphi_b(b)=f(q(b))$ is defined. With $e=q(b)$ the
displayed equality gives
$\varphi_e=\varphi_{\varphi_b(b)}=\varphi_{f(e)}$.
\end{proof}
The fixed point is \emph{extensional}: the two indices describe the
same partial function. It need not satisfy $f(e)=e$.
For example, $f(n)=n+1$ has no numerical fixed point, yet the theorem
still supplies indices for which $\varphi_e=\varphi_{e+1}$.
No decision procedure for equality of arbitrary programs is asserted.

\begin{example}[An index which outputs itself]\label{ex:self-output}
By $s$-$m$-$n$ there is a total computable compiler $c$ for which
$\varphi_{c(e)}(x)=e$ for every $x$. Apply
Theorem~\ref{thm:kleene} to $c$. It supplies $q$ with
$\varphi_q=\varphi_{c(q)}$, hence
\[
 \forall x\qquad\varphi_q(x)=q.
\]
This is self-output at the level of numerical program indices;
it does not require $q=c(q)$.
A source-level illustration, in Python, is
\begin{verbatim}
s='s=%r;print(s%%s)';print(s%s)
\end{verbatim}
Its string template contains a place for its own quoted string.
The expression \texttt{s\%s} fills that place, while
\texttt{\%\%} in the template contributes a literal percent sign.
The output is exactly the displayed source text with a final newline,
without opening the source file. This illustration and numerical
self-output use different encodings; neither should be confused
with equality of an index and its transformed index. See
\cite{Terwijn2024} and Online source~\ref{online:waterloo} for the
recursion-theoretic context.
\end{example}

\subsection{Parameters and effective choice}
\begin{corollary}[A parameterized form]\label{cor:parameterized}
For every total computable $F:\mathbb N^2\to\mathbb N$,
there is a total computable $r$ such that
\[
 \forall n\qquad
 \varphi_{r(n)}=\varphi_{F(r(n),n)}.
\]
\end{corollary}
\begin{proof}
Keep the function $q$ from the preceding proof. By parameterization,
choose a total computable $b(n)$ giving an index for the unary total
function $x\mapsto F(q(x),n)$. Set $r(n)=q(b(n))$.
For each $n$, $\varphi_{b(n)}(b(n))=F(r(n),n)$ is defined,
so the defining property of $q$ proves the equality.
\end{proof}
This proof constructs the selector $r$; it does not infer computability
from the quantifier pattern $\forall n\,\exists e$ alone.
More general fixed-point results in computability theory need not
retain this effective parameterized conclusion
\cite[Section~8]{Terwijn2024}.

\subsection{Numberings and the role of precompleteness}
A numbering is a surjection $\gamma:\mathbb N\to S$ onto some
set of denotations. It induces an equivalence relation
$n\sim_\gamma m$ if $\gamma(n)=\gamma(m)$.
It is \emph{precomplete} if every partial computable function
$\psi$ admits a total computable $t$ such that
\[
 \psi(n)\downarrow\quad\Longrightarrow\quad
 t(n)\sim_\gamma\psi(n).
\]
The totalization is modulo denotational equality, not necessarily
an extension as a numerical function. Terwijn describes Ershov's
recursion theorem: for a precomplete numbering and every total
computable $f$, there is $e$ with $f(e)\sim_\gamma e$
\cite[Section~4]{Terwijn2024}.

For the numbering $\gamma(e)=\varphi_e$, precompleteness follows
by compiling, for each $n$, the program which first attempts to
compute $\psi(n)$ and then simulates the resulting program index.
The compiler is total by $s$-$m$-$n$ even when the compiled program
runs forever. This explains how partial evaluation can support
an extensional totalization without a halting decider.

These results share a diagonal mechanism with Lawvere's theorem,
but the present proof of the recursion theorem uses acceptability
and partial simulation explicitly. A categorical reformulation
must encode those structures and their denotational equality;
substituting the universal partial evaluator into a theorem about
total set functions would erase the decisive hypothesis.

\section{When a uniform fixed-point operator is available}\label{sec:uniform}
The conclusion of Corollary~\ref{cor:ccc} has quantifier form
``for every $g$, there exists $b$''. It does not by itself
construct a morphism that chooses fixed points uniformly as
$g$ varies. The following stronger, explicitly conditional
construction explains one route to such a morphism.

\begin{proposition}[A split representation gives an operator]\label{prop:uniform}
Let $D$ be an object of a CCC and put $H=D^D$, with
$\ev:H\times D\to D$. Suppose there are morphisms
\[
 e:D\to H,\qquad s:H\to D,\qquad e\circ s=\id_H.
\]
Then there is a morphism $Y:H\to D$ such that
\begin{equation}\label{eq:uniform}
 \ev\circ\pair{\id_H}{Y}=Y.
\end{equation}
In particular, for every endomorphism $g:D\to D$, the point
$Y\circ\widehat g:1\to D$ is fixed by $g$.
\end{proposition}
\begin{proof}
Define the following morphisms in order:
\[
\begin{aligned}
 \delta&=\ev\circ\pair{e}{\id_D}:D\to D,\\
 k&=\ev\circ\pair{\pi_H}{\delta\circ\pi_D}:H\times D\to D,\\
 \ell&=\Lambda k:H\to H,\\
 a&=s\circ\ell:H\to D,\qquad Y=\delta\circ a:H\to D.
\end{aligned}
\]
Here $\ell$ is the transpose of the two-variable morphism
$k$. Since $e\circ a=\ell$, the defining equality of currying gives
\begin{align*}
 Y&=\delta\circ a
   =\ev\circ\pair{e\circ a}{a}
   =\ev\circ\pair{\ell}{a}\\
  &=k\circ\pair{\id_H}{a}
   =\ev\circ\pair{\id_H}{\delta\circ a}
   =\ev\circ\pair{\id_H}{Y}.
\end{align*}
These are equalities of morphisms $H\to D$, without a
point-detection assumption. Precompose with
$\widehat g:1\to H$ and use~\eqref{eq:name} to obtain
$g\circ Y\circ\widehat g=Y\circ\widehat g$.
\end{proof}

The split representation is stronger than point-surjectivity:
for a map $h:D\to D$, $a=s\circ\widehat h$ satisfies
$e\circ a=\widehat h$. We do not assert that every CCC has
such a pair $e,s$, or that the existence of a representing
point for each map automatically supplies $s$.

\begin{remark}[The limits of a programming analogy]
Equation~\eqref{eq:uniform} is a categorical fixed-point
operator equation under the stated retraction hypothesis.
It helps explain why fixed-point combinators are relevant to
self-application. It does not establish a fixed-point
combinator at every type of an arbitrary typed language.
Already in $\Set$, the endomorphism exchanging the two elements
of $2$ has no fixed point. A proposed programming model must
therefore justify its particular domain, representation
structure, and notion of computation.
\end{remark}

\section{Positive models and the strength of fixed-point operators}\label{sec:models}
A useful survey should display a nontrivial representation hypothesis,
not only use its failure to derive contradictions. Order-theoretic
models restrict the permitted functions and allow information to grow
by directed approximation. This changes both the function space and
the meaning of self-application.

\subsection{Directed completeness and continuous functions}
A nonempty subset $S$ of a partially ordered set is \emph{directed}
if every two members have an upper bound in $S$. A \emph{dcpo} is a
partially ordered set in which every directed subset has a supremum.
It is \emph{pointed} when it has a least element $\bot$.
A function between dcpos is \emph{Scott-continuous} here if it is
monotone and preserves suprema of directed subsets. It is not
required to preserve the least element.

Dcpos and Scott-continuous functions form a cartesian closed category.
Products carry the componentwise order; the function space $[A,B]$
consists of Scott-continuous functions, ordered pointwise.
A directed family of functions has its supremum pointwise.
The pointwise supremum is again Scott-continuous: for a directed
$S\subseteq A$, one can interchange the two directed suprema in
$\sup_i\sup_{a\in S}f_i(a)=\sup_{a\in S}\sup_i f_i(a)$.
For a Scott-continuous $u:X\times A\to B$, each map
$a\mapsto u(x,a)$ is continuous, and the curried map is continuous
because, pointwise in $a$,
$u(\sup_i x_i,a)=\sup_i u(x_i,a)$.
Evaluation is Scott-continuous as well. To see the main point
for evaluation, a directed family of pairs $(f_i,a_i)$ gives
\[
 (\sup_i f_i)(\sup_i a_i)=\sup_{i,j}f_i(a_j)
                         =\sup_i f_i(a_i).
\]
The last equality uses directedness to dominate any pair of indices
by one index. This calculation also explains why unrestricted
set functions are the wrong exponential for this category.
The fixed-point operators associated with domain models and their
algebraic properties are discussed by Hasegawa
\cite[Section~3.4]{Hasegawa2009}.

\subsection{An explicit representation on a powerset dcpo}
Let $D=\mathcal P(\mathbb N)$, ordered by inclusion. Directed
suprema are unions and $\bot=\varnothing$. Write
$\mathcal P_{\mathrm{fin}}(\mathbb N)$ for the finite subsets of
$\mathbb N$. Fix a bijective coding
\[
 \kappa:\mathcal P_{\mathrm{fin}}(\mathbb N)\times\mathbb N
              \longrightarrow\mathbb N.
\]
Such a coding exists because its domain is countably infinite.
Define an application operation by
\begin{equation}\label{eq:graph-application}
 a\mathbin{\cdot}b
 =\{n:\exists u\in\mathcal P_{\mathrm{fin}}(\mathbb N),\
               u\subseteq b\text{ and }\kappa(u,n)\in a\}.
\end{equation}
The set $a$ records finite pieces of input information sufficient
to produce an output number. The input $b$ need not itself be finite.
\begin{proposition}[A powerset representation]\label{prop:powerset-model}
With $H=[D,D]$, the maps
\[
 e:D\to H,\quad e(a)(b)=a\mathbin{\cdot}b,
 \qquad
 s:H\to D,\quad
 s(f)=\{\kappa(u,n): u\text{ finite and } n\in f(u)\}
\]
are Scott-continuous and satisfy $e\circ s=\id_H$.
\end{proposition}
\begin{proof}
Application is monotone in both arguments. To prove preservation
of a directed supremum of pairs $(a_i,b_i)$, suppose
$n\in(\bigcup_i a_i)\mathbin{\cdot}(\bigcup_i b_i)$.
Choose a finite witness $u$. The code $\kappa(u,n)$ belongs to
one $a_i$ and each member of $u$ belongs to some $b_j$.
Directedness supplies a common index above these finitely many
indices. At that index the same witness establishes membership
in $a_k\mathbin{\cdot}b_k$. The opposite inclusion is monotonicity.
Thus application is jointly Scott-continuous, and its transpose $e$
is a Scott-continuous map into $H$.

For a directed family $(f_i)$ in $H$, pointwise suprema give
\[
 \kappa(u,n)\in s(\sup_i f_i)
 \ \Longleftrightarrow\ n\in\bigcup_i f_i(u)
 \ \Longleftrightarrow\ \kappa(u,n)\in\bigcup_i s(f_i).
\]
The same definition shows monotonicity, so $s$ is Scott-continuous.
Finally, the finite subsets of any $b\subseteq\mathbb N$ are a
directed family with union $b$. Continuity of $f$ gives
\[
 e(s(f))(b)
  =\bigcup_{u\subseteq b,\ u\text{ finite}}f(u)=f(b).
\]
This proves the retraction identity as an equality of functions.
\end{proof}
\begin{example}[Two instructions with numerical codes]\label{ex:numeric-codes}
The coding in~\eqref{eq:graph-application} can be made explicit.
For a finite $u\subseteq\mathbb N$, put
\[
 c(u)=\sum_{i\in u}2^i,\qquad
 \pi(p,q)=\frac{(p+q)(p+q+1)}2+q,\qquad
 \kappa(u,n)=\pi(c(u),n).
\]
Binary coding bijects finite subsets with natural numbers, and the
pairing $\pi$ bijects $\mathbb N^2$ with $\mathbb N$.
In particular, $\kappa(\varnothing,4)=14$ and
$\kappa(\{1,2\},7)=98$. Let $a=\{14,98\}$. Its instructions
say ``always output $4$'' and ``output $7$ once the input contains
both $1$ and $2$''. Therefore
\[
\begin{array}{c|cccc}
 b&\varnothing&\{1\}&\{1,2\}&\{1,2,3\}\\\hline
 a\mathbin{\cdot}b&\{4\}&\{4\}&\{4,7\}&\{4,7\}
\end{array}
\]
Increasing the input does not remove an output. Each output already
has a finite witness; this is the concrete mechanism behind
Scott continuity in Proposition~\ref{prop:powerset-model}.
\end{example}

\begin{example}[Identity, constants, and different representations]
For $I(b)=b$, the specified section gives
\[
 s(I)=\{\kappa(u,n):u\text{ finite},\ n\in u\},\qquad
 e(s(I))(b)=\bigcup_{u\subseteq b,\ u\text{ finite}}u=b.
\]
For $b=\{1,3\}$, the witnesses $\{1\}$ and $\{3\}$ produce
its two members. For the constant map $K(b)=\{4\}$,
\[
 s(K)=\{\kappa(u,4):u\text{ finite}\}.
\]
The smaller set $a_0=\{\kappa(\varnothing,4)\}=\{14\}$ also
represents $K$. Indeed, $a_0\mathbin{\cdot}b=\{4\}$ for every
$b$. The representations are different: for example,
$\kappa(\{0\},4)\in s(K)$ but is not in $a_0$.
A specified section chooses a representation without making all
representations unique.
\end{example}

This explicit construction realizes the hypothesis of
Proposition~\ref{prop:uniform}, with a genuinely non-singleton
object $D$. It is compatible with the set-theoretic boundary because
$H$ contains the Scott-continuous maps, not all set functions
$D\to D$. The representation table is
$E(a,b)=a\mathbin{\cdot}b$. It represents every permitted map
using the row $s(f)$.

\subsection{Least fixed points are a stronger order-theoretic conclusion}
\begin{proposition}[The increasing approximation construction]
If $D$ is a pointed dcpo and $f:D\to D$ is Scott-continuous,
then
\[
 \mu f=\sup_{n\geq0} f^n(\bot)
\]
is the least fixed point of $f$.
\end{proposition}
\begin{proof}
Since $\bot\leq f(\bot)$ and $f$ is monotone, the iterates form
an increasing chain. Continuity yields
$f(\mu f)=\sup_n f^{n+1}(\bot)=\mu f$.
If $x=f(x)$, induction from $\bot\leq x$ gives
$f^n(\bot)\leq x$ for every $n$, whence $\mu f\leq x$.
\end{proof}
Lawvere's basic theorem has no order or leastness in its statement.
The powerset model supports both a representation-based operator
and the increasing-approximation construction; we have not proved
that an arbitrary representation-based operator selects the least
fixed point. Such an identification requires a further argument.

\begin{example}[Growing a fixed point, and choosing the least one]
For the continuous map
$F(b)=\{0\}\cup\{n+1:n\in b\}$, the approximation sequence is
\[
 \varnothing\longmapsto\{0\}\longmapsto\{0,1\}
 \longmapsto\{0,1,2\}\longmapsto\cdots.
\]
Induction gives $F^k(\varnothing)=\{0,\ldots,k-1\}$ for
$k\geq1$, and the supremum is $\mathbb N$.
For comparison, $G(b)=b\cup\{0\}$ has as fixed points
\emph{all} subsets containing $0$. Its approximation sequence
stabilizes at $\{0\}$, the least fixed point. The other solutions
$\{0,1\}$ and $\mathbb N$ illustrate why fixed-point existence,
uniqueness, and leastness are separate questions.
\end{example}

\begin{example}[Factorial by successive approximation]\label{ex:factorial}
Let $D=\mathbb N_\bot$ have the flat order: $\bot\leq n$ for
all numbers, and different numbers are incomparable. In the
pointed dcpo $H=[D,D]$, define
\[
 \Phi(f)(\bot)=\bot,\qquad \Phi(f)(0)=1,\qquad
 \Phi(f)(n+1)=(n+1)\,f(n).
\]
Multiplication here is strict in the lifted value: a numerical
factor times $\bot$ is $\bot$. Each $\Phi(f)$ is strict and
continuous on the flat dcpo; $\Phi$ is Scott-continuous because
at each numerical input it uses one coordinate of $f$ and a
continuous lifted multiplication.
Starting with the constantly undefined $f_0$ and putting
$f_{k+1}=\Phi(f_k)$ gives
\[
\begin{array}{c|cccc}
 &0&1&2&3\\\hline
 f_0&\bot&\bot&\bot&\bot\\
 f_1&1&\bot&\bot&\bot\\
 f_2&1&1&\bot&\bot\\
 f_3&1&1&2&\bot\\
 f_4&1&1&2&6
\end{array}
\]
Induction shows that $f_k(n)=n!$ for $n<k$ and $f_k(n)=\bot$
for $n\geq k$. All the functions send $\bot$ to $\bot$.
Their pointwise supremum is thus factorial on natural numbers,
with value $\bot$ at $\bot$, and is the least fixed point of
$\Phi$. The approximation depth measures how many recursive calls
have been unfolded. Undefined information is not the number $0$.
For a teaching account of domain-based recursion, see Online
source~\ref{online:horwitz}; the displayed factorial calculation
is worked out here using the preceding definitions.
\end{example}

\subsection{Parameterized operators, Conway identities, and trace}
For a category $\mathcal C$ with finite products, a parameterized
fixed-point operator assigns to every $f:A\times X\to X$ a
morphism $f^\dagger:A\to X$ satisfying
\begin{equation}\label{eq:dagger}
 f^\dagger=f\circ\pair{\id_A}{f^\dagger}.
\end{equation}
It is natural in the parameter when, for $u:B\to A$,
\[
 (f\circ(u\times\id_X))^\dagger=f^\dagger\circ u.
\]
A fixed-point equation and parameter naturality still do not comprise
all the coherence conditions used in recursion semantics.
A \emph{Conway operator} additionally satisfies dinaturality and
an iterated-diagonal identity. For instance, for
$f:A\times X\times X\to X$, the latter says
\[
 (f\circ(\id_A\times\Delta_X))^\dagger
                  =(f^\dagger)^\dagger,
\]
where product associativity is understood and the first dagger on the
right removes the last $X$ parameter. These identities govern how
nested and simultaneous recursive equations interact.

A trace instead closes a feedback wire in a map
$g:A\times X\to B\times X$, producing
$\operatorname{Tr}^X_{A,B}(g):A\to B$, subject to its trace axioms.
The Hyland--Hasegawa correspondence identifies Conway operators
with traces when the monoidal product is the categorical product
\cite{Hasegawa1997,Hasegawa2009}. One direction of the correspondence
has the particularly useful formula
\[
 f^\dagger=\operatorname{Tr}^X_{A,X}(\Delta_X\circ f).
\]
The map being traced has type $A\times X\to X\times X$;
its two output components play the observable-output and feedback roles.
We cite the correspondence rather than prove all its axioms here.

The distinction can now be read from the types. A global fixed point
is a map $1\to X$ for a single endomorphism. An operator
$Y:X^X\to X$ provides a selection with an internal parameter.
A dagger is a family for arbitrary parameter objects, and a Conway
operator imposes equations on that family. A least fixed point
adds an order comparison with every other solution. None of these
extra conclusions should be attributed to the basic existence theorem
without their additional structure.

\begin{example}[Two Boolean equations]\label{ex:boolean-recursion}
Order $\{0,1\}$ by $0\leq1$. For parameters $p,q\in\{0,1\}$,
consider the monotone map
$F_{p,q}(x,y)=(p\vee y,q\wedge x)$ on the four-element product
lattice. Its fixed points are the solutions of
\[
 x=p\vee y,\qquad y=q\wedge x.
\]
\[
\begin{array}{cc|c|c}
 p&q&\text{all fixed points}&\text{least fixed point}\\\hline
 0&0&(0,0)&(0,0)\\
 0&1&(0,0),(1,1)&(0,0)\\
 1&0&(1,0)&(1,0)\\
 1&1&(1,1)&(1,1)
\end{array}
\]
Iteration from $(0,0)$ gives the least solution in each row.
For $p=q=1$, the sequence is $(0,0),(1,0),(1,1)$ and then
stabilizes. For $p=0,q=1$, the sequence stays at $(0,0)$ despite
the existence of the other fixed point $(1,1)$.

One can also solve $y=q\wedge x$ with $x$ as parameter, then
solve the outer equation $x=p\vee(q\wedge x)$ by its least
fixed point, and finally substitute back for $y$. The same four
least solutions result. This is a small illustration of the
simultaneous-to-nested recursion described by the Beki\'c property
\cite[Section~3.4]{Hasegawa2009}. It is not a proof of the general
Conway identities, nor a claim that this finite lattice has a
Lawvere representation of its whole endomorphism space.
\end{example}

\section{Weakening the structure and changing the representation}\label{sec:generalizations}
The diagonal proof is short enough to reveal which assumptions it
actually uses. Exponentials express the representation hypothesis
conveniently but are not used after the table has been uncurried.
Finite products provide diagonals and pairing. Even these can be
replaced by specified operations, provided the compatibility used
at the selected point is retained.

\subsection{A single representing row can suffice}
For a fixed $g:B\to B$, the proof of
Theorem~\ref{thm:lawvere} needs only one representing row: the row
for $g\circ E\circ\Delta_A$. It does not inspect every morphism
$A\to B$ after that row has been chosen. Full weak row-surjectivity
is a useful sufficient hypothesis because it guarantees such a row
for every $g$.

This observation is also constructively useful. If one supplies the
appropriate row as data, the proof directly constructs a fixed point.
A statement that the row merely exists, especially if existence is
propositionally truncated, has a different type from a function
selecting such data uniformly. The distinction will reappear in
Section~\ref{sec:constructive}.

\subsection{A magmoidal formulation}
Following Roberts \cite{Roberts2023}, let $\#$ be a bifunctor
$\mathcal C\times\mathcal C\to\mathcal C$, choose an object $t$,
and suppose there are morphisms
$\delta_X:X\to X\#X$ natural in $X$.
Associativity, a monoidal unit, and projection maps are not part of
these assumptions. A map $a:t\to A$ is a $t$-point.
For $F:A\#A\to C$ and $\sigma:C\to C$, the required row equation
can be stated without referring to an exponential.

\begin{theorem}[A local diagonal form, after Roberts]\label{thm:magmoidal}
Suppose there is a $t$-point $a_0:t\to A$ such that, for every
$t$-point $a:t\to A$,
\[
 \sigma\circ F\circ\delta_A\circ a
             =F\circ(a_0\#a)\circ\delta_t.
\]
Then $c=F\circ\delta_A\circ a_0:t\to C$ satisfies
$\sigma\circ c=c$.
\end{theorem}
\begin{proof}
Use the hypothesis with $a=a_0$ and naturality at $a_0$:
\[
 \sigma\circ c
 =F\circ(a_0\#a_0)\circ\delta_t
 =F\circ\delta_A\circ a_0=c.
\]
\end{proof}
This is the form of Roberts's Theorem~14. Only the displayed
compatibility at $a_0$ is needed for this final calculation;
requiring natural diagonals everywhere provides a reusable setting.
The usual cartesian statement takes $\#=\times$ and $t=1$.
If $t$ is not terminal, the conclusion is a fixed $t$-point,
not automatically a global fixed point.

One illustrative example is the category of countably infinite sets
and finite-to-one functions, with $t=\mathbb N$ as the designated
object. Cartesian product of underlying sets
is a bifunctor in this category and the ordinary diagonal is
finite-to-one. However, a projection $A\times B\to A$ has an
infinite fiber when $B$ is infinite, so it is not a morphism.
The underlying set product therefore cannot be used as the categorical
product with its usual projections. Diagonal reasoning survives,
while the standard cartesian hypotheses do not describe this
particular structure \cite{Roberts2023}.

The lesson is narrower than saying that every noncartesian category
supports diagonalization. The specified bifunctor and compatible
diagonal maps still do real work. In particular, a theory without
an appropriate duplication operation cannot use this proof merely
because it has a tensor product.

\subsection{Extensions and converse questions}
Soto-Andrade and Varela \cite{SotoVarela1984} study extensions
of Lawvere's result and investigate whether fixed-point behavior
can be explained by a retraction from a higher reflexive structure.
Their paper reports a result in a setting of chain-complete posets
and continuous morphisms. This historical direction is different
from assuming a split representation and then obtaining fixed points:
it asks whether fixed-point behavior implies a suitable representation
in a larger structure. We do not state an unrestricted converse for
arbitrary categories, nor reproduce the full hypotheses of that result.

For comparison, Proposition~\ref{prop:powerset-model} explicitly
constructs a retraction in one model and
Proposition~\ref{prop:uniform} proves the forward implication.
A converse must specify what larger structure is allowed and which
maps count as retractions. Without those choices, ``having fixed
points means being self-representing'' is not a theorem statement.

\subsection{Relational and interactive self-reference}
The evaluation table can also be replaced by a relation.
In $\Set$, let $R\subseteq X\times X$ and suppose every subset
$P\subseteq X$ is a row
$R_c=\{x:(c,x)\in R\}$ for some $c\in X$.
For any operation $\tau:\mathcal P(X)\to\mathcal P(X)$,
choose $c$ representing
$P=\tau(\{x:(x,x)\in R\})$. Then
\[
 (c,c)\in R\quad\longleftrightarrow\quad
 c\in\tau(\{x:(x,x)\in R\}).
\]
For complementation, the right side is exactly $(c,c)\notin R$,
so no relation can represent every subset in this way. For a general
$\tau$, this displayed equation alone does not claim that $P$ is a
fixed point of $\tau$ as an operation on $\mathcal P(X)$.
That stronger reading would confuse a truth value at $c$ with a whole
subset of $X$.

Abramsky and Zvesper \cite{AbramskyZvesper2015} give a relational
Lawvere fixed-point lemma in a regular category. Its representation
condition is their very weak point-surjectivity of a relation, and
its operation is an endomorphism of the subobject functor:
\[
 \tau:\operatorname{Sub}\Longrightarrow\operatorname{Sub}.
\]
It is a natural transformation, not an arbitrary function on one
powerset. Their conclusion is a fixed point at the level of this
functorial predicate structure. More precisely, the fixed natural
family is determined by a subobject $p$ of the terminal object,
with $\tau_1(p)=p$ in $\operatorname{Sub}(1)$. Naturality controls how the operation
commutes with substitution; it replaces an assumption that would
otherwise be easy to overlook in a predicate calculation.

They apply this setting to the Brandenburger--Keisler paradox,
where one agent's beliefs concern another agent's assumptions.
The reduction uses specified representation and nonemptiness
conditions on belief relations. The regular-category formulation
supports compositional and multi-agent extensions. We refer to
that work for these conditions and proofs; a slogan about
``two-person Russell paradox'' is not a replacement for them.
The resulting application broadens the survey beyond numerical
self-application while keeping the representation hypothesis visible.

\section{Constructive interpretations and proof data}\label{sec:constructive}
Type theory makes a useful distinction between having a witness
and knowing that one exists. Let $A,B$ be types and let
$r:A\to(A\to B)$. A pointwise section is a function
$s:(A\to B)\to A$ equipped with identities
\[
 \forall h:A\to B\ \forall x:A,\qquad r(s(h))(x)=h(x).
\]
For $g:B\to B$, set
\[
 h_g(x)=g(r(x)(x)),\qquad a_g=s(h_g),\qquad
 \operatorname{lfix}(g)=r(a_g)(a_g).
\]
The section identity at $x=a_g$ supplies a path
$\operatorname{lfix}(g)=g(\operatorname{lfix}(g))$.
The construction needs pointwise identities rather than an identity
of functions; function extensionality is not needed for this calculation.

\subsection{Witnesses versus truncated existence}
The relevant types have different informational content:
\[
 \prod_{g:B\to B}\sum_{b:B}(g(b)=b),\qquad
 \prod_{g:B\to B}\left\|\sum_{b:B}(g(b)=b)\right\|.
\]
Here $\|X\|$ denotes propositional truncation, which retains
existence but discards the specified witness. A supplied pointwise
section gives the first type. If the representation condition only
states truncated existence of a row for each $h$, the construction
can be mapped under truncation to give the second type. It does
not in general extract a fixed-point selector in the first type.
The truncation eliminator permits elimination into propositions,
not unrestricted extraction of witness data.

Escard\'o's commented Agda development of Lawvere's theorem
makes these variants explicit, starting with sections and then
considering surjections and additional hypotheses. See Online source~\ref{online:escardo}, in particular
\texttt{has-section\textperiodcentered}, \texttt{lfix}, and
\texttt{lfix-is-fixed-point}.
It offers a direct source for discussing the proof terms and
constructive distinctions. The type-theoretic formulation does
not make a split representation of every function type available.
For instance, a type with a fixed-point-free negation operation
still obstructs the relevant representation condition.

\subsection{What a formalization actually establishes}
A machine-checked formulation has a particular universe level,
notion of equality, and list of hypotheses. A type-level proof
for $r:A\to(A\to B)$ is not, merely by its existence, a
formalization of every categorical and logical application in
this survey. Likewise, a derivation of fixed-point existence
under truncation does not certify a specified uniform operator.
When comparing formal treatments, the statement and its dependencies
must be compared before the verification result is transferred.

These distinctions provide a more direct connection with type theory
than an unrestricted claim that Lawvere's theorem produces recursive
types or identity types. The categorical proof and the proof term
agree on the decisive step: applying the row identity at the same
point which represents the diagonal composite. Their surrounding
languages differ in what it means to give that point.

\section{Related literature and the scope of applications}\label{sec:literature}
\subsection{Topological variants}
McCallum \cite{McCallum2020} studies a connection with Brouwer's
fixed-point theorem using a generalized Lawvere-type setting.
The formulation involves spaces $A'$ and $A''$ with the same
underlying set, with the topology on $A'$ finer than that on
$A''$, and a continuous surjection $A'\to X^{A''}$.
In the categorical formulation, the map $j:A'\to A''$ is
required to be bijective on global points, and the relevant
exponentials must exist in the specified setting. For each
object $X$ in the chosen class and its representation
$q:A'\to X^{A''}$, an additional descent condition is
required: every $u:A'\to X$ that factors through
$\pair{q}{j}:A'\to X^{A''}\times A''$ must also factor
through $j$. This condition permits the one-variable map
used in the diagonal argument to descend to $A''$; it is not
implied merely by the existence of a continuous surjection.
The paper also treats an obstruction to a direct
construction of a continuous surjection $A\to[0,1]^A$ when
the relevant exponential topology exists. Accordingly, this
work should not be described as an unconditional application
of Corollary~\ref{cor:ccc} in the category of all topological
spaces. That category is not cartesian closed in general;
the spaces, function-space topology, and representation
hypothesis must be specified. This survey records the scope
of McCallum's proposed connection and does not reproduce the
proofs of its topological results.

\subsection{Real-cohesive type theory and Brouwer}
Shulman \cite{Shulman2018} combines homotopy type theory with
cohesive structure and a refinement called real-cohesion to prove
Brouwer's fixed-point theorem. His treatment distinguishes continuous
paths from the identifications of homotopy type theory and uses
modalities and crisp variables. These are additional structures;
they are not hypotheses already present in
Corollary~\ref{cor:ccc}. Roberts discusses a related modal perspective
in his study of fixed-point operators, while explicitly not claiming
to unify the two classes of fixed-point theorem \cite{Roberts2023}.

McCallum's modified representation setting and Shulman's real-cohesive
type theory are different routes to study connections with topology.
We do not identify their function spaces or derive one construction
from the other. They provide a comparison of scope rather than an
additional corollary of the elementary diagonal proof.

\subsection{Identity types as related background}
Awodey and Warren \cite{AwodeyWarren2009} develop
homotopy-theoretic models of identity types using model
categories, weak factorization systems, and path objects.
This is relevant background for categorical semantics of
type theory. It is not used here as a proof that Lawvere's
fixed-point theorem constructs identity types, establishes
univalence, or yields recursive types. A direct connection
would require a separate statement and a construction
verifying its hypotheses. No such connection is established
in this survey.

\subsection{A different notion of Lawvere completeness}
Clementino and Hofmann \cite{ClementinoHofmann2009} study
Lawvere completeness for enriched topological structures.
Their paper relates the notion to weak sobriety and Cauchy
completeness, following a different line of Lawvere's work.
The occurrence of Lawvere's name does not make this
completeness notion an application of the diagonal theorem
proved above. It is included as related categorical
background, with the distinction made explicit.

\section{A comparison of the conclusions}\label{sec:comparison}
The phrase ``fixed-point theorem'' covers several different outputs.
Table~\ref{tab:comparison} provides a guide to the forms treated here.
Each row has its own hypothesis; arrows between rows require proof.

\begin{table}[htbp]
\centering\small
\begin{tabular}{>{\raggedright\arraybackslash}p{.23\linewidth}>{\raggedright\arraybackslash}p{.34\linewidth}>{\raggedright\arraybackslash}p{.32\linewidth}}
\toprule
Setting & Essential additional data & Conclusion \\
\midrule
Finite-product category & Weak row representation for $E$ & A global point fixed by each endomorphism \\
CCC with a split representation & $e:D\to D^D$, $s:D^D\to D$, $es=\id$ & A morphism $Y:D^D\to D$ satisfying the evaluation equation \\
Pointed dcpo & Order, bottom, Scott-continuous map & A least fixed point from increasing approximations \\
Arithmetic & Coding and representable substitution & A sentence provably equivalent to its prescribed code predicate \\
Acceptable program numbering & Partial evaluator and $s$-$m$-$n$ & An extensional program fixed point, with an effective parameterized form \\
Magmoidal category & Bifunctor, compatible diagonals, designated $t$ & A fixed $t$-point from the required row \\
Constructive type theory & Pointwise section, or truncated row existence & A specified fixed point, or a truncated existence statement, respectively \\
\bottomrule
\end{tabular}
\caption{Representation, structure, and the kind of fixed point obtained.}
\label{tab:comparison}
\end{table}

Three checks help when moving between these settings.
First, determine the permitted maps. Complementation is a set
function on a powerset but is not a Scott-continuous endomorphism.
Second, determine equality: a path, a morphism equality, provable
logical equivalence, or equality of partial functions can play
different roles. Third, determine the available choice data:
existence of a representing row need not give a section or an
effective selection of program indices.

A map described as ``evaluation'' must have the claimed type and
belong to the specified class of morphisms. The diagonal composite
must belong to the family being represented. The selected row must
represent that composite with the required quantifier order.
Only then can the final self-evaluation step be used.

\section{Further directions and limits of the survey}\label{sec:directions}
The literature suggests several directions for further study.
They are directions, not assertions that the corresponding problems
are unsolved in full generality.

\paragraph{Representation versus the fixed-point property.}
Soto-Andrade--Varela investigate how fixed-point behavior can arise
from a retraction involving a reflexive structure. For any proposed
converse, one should specify the category, the allowed enlargement,
and whether the retraction represents maps uniformly. A result in
one order-theoretic category does not imply a universal converse.

\paragraph{Minimal structural assumptions.}
Roberts shows that much of the usual cartesian structure can be
removed while retaining compatible diagonals. A further comparison
is to identify the logical structural rules interpreted by a
particular model, and whether only local duplication at designated
points is needed. Dropping duplication altogether requires a
different argument; it is not a consequence of the theorem here.

\paragraph{Effective and uniform fixed points.}
Computability theory distinguishes the existence of an index from
an effective selector. Categorical semantics distinguishes global
existence from an internal operator and from Conway coherence.
Comparing these distinctions may clarify applications, but no general
equivalence between them has been established in this survey.

\paragraph{Topology and constructive assumptions.}
The topological function space, continuity of evaluation, descent of
the diagonal composite, and the logical strength of existence all
matter in comparing Brouwer-related formulations. The existence of
a topology on the same underlying set is not enough to transfer a
representation theorem. The cited approaches should be assessed
under their stated assumptions.

\paragraph{Scope of the literature selection.}
This survey emphasizes the diagonal mechanism and a small set of
explicit applications and extensions. It does not provide a complete
history of every paradox, a classification of domain models, or proofs
of all topological and relational results cited. It also does not
establish that each research question raised in an older paper
remains open. The recent code-based treatment of arithmetic and the
survey of computability theory give a path beyond the foundational
sources without overstating a new theorem.

\section{Conclusion}
The common mechanism is a representation followed by self-evaluation.
The basic categorical proof uses a row for a diagonal composite and
checks it at its own representing point. The hypotheses determine
which maps can be represented and what kind of equality the
self-evaluation yields.

The applications exhibit both sides of the mechanism.
Fixed-point-free operations obstruct excessive representations, as
in Cantor's theorem and the total evaluation or halting examples.
Arithmetic coding constructs provable self-reference, after which
truth and proof predicates lead to distinct limitation results.
Partial program evaluation supports extensional recursion, and the
powerset dcpo gives an explicit positive representation of continuous
maps. A split representation constructs an internal fixed-point
operator, while order, parameterization, and coherence support
stronger notions of solution.

Replacing products by a bifunctor still requires compatible diagonals.
Relational self-reference requires predicate structure and naturality.
Constructive existence distinguishes a witness from its truncation.
Lawvere's theorem is a useful organizing principle because it exposes
a reusable calculation together with the representation obligations
that each application must meet.

\appendix
\section{Three universal-property checks}\label{app:checks}
\subsection{Terminal objects are unique up to unique isomorphism}
\begin{lemma}
If $T$ and $T'$ are terminal objects, there is a unique
isomorphism $u:T\to T'$.
\end{lemma}
\begin{proof}
Terminality gives unique maps $u:T\to T'$ and $v:T'\to T$.
The composites $v\circ u$ and $u\circ v$ equal the relevant
identity morphisms by terminality. Every map $T\to T'$ is
$u$, so the isomorphism is unique.
\end{proof}

\subsection{Both product projections are required}
\begin{lemma}
For $p,q:X\to A\times B$, if
$\pi_A\circ p=\pi_A\circ q$ and
$\pi_B\circ p=\pi_B\circ q$, then $p=q$.
\end{lemma}
\begin{proof}
Both maps are the unique pairing of their common components.
\end{proof}
The hypothesis involving both projections cannot be replaced
by one projection in general. In sets, the two maps from a
singleton to $\{*\}\times\{0,1\}$ selecting $(*,0)$ and
$(*,1)$ have the same first component and different second
components. Universal properties determine the relevant
whole cone or factorization, not an arbitrary component of
it.

\subsection{The evaluation map is part of the structure}\label{app:exponential}
\begin{lemma}
If $(H,\epsilon)$ and $(H',\epsilon')$ are exponentials of
$B$ by $A$, there is a unique isomorphism $u:H\to H'$ such
that
\[
 \epsilon'\circ(u\times\id_A)=\epsilon.
\]
\end{lemma}
\begin{proof}
The universal property of $(H',\epsilon')$ gives a unique
$u:H\to H'$ with that equality. The other exponential
structure gives $v:H'\to H$ satisfying
$\epsilon\circ(v\times\id_A)=\epsilon'$.
By composition of product morphisms,
\[
 \epsilon\circ((v\circ u)\times\id_A)=\epsilon.
\]
Uniqueness of the transpose of $\epsilon$ with respect to
$(H,\epsilon)$ makes $v\circ u=\id_H$. The other composite
is likewise $u\circ v=\id_{H'}$.
\end{proof}
This is uniqueness up to the structure-preserving
isomorphism, not literal equality of evaluation maps on a
fixed underlying object. For example, with $A=\{*\}$ and
$B=\{0,1\}$, either map
\[
 \epsilon(t,*)=t,\qquad \epsilon'(t,*)=1-t
\]
on the underlying object $H=B$ satisfies the exponential
universal property. For the second structure the transpose
of $u:X\times A\to B$ sends $x$ to $1-u(x,*)$.
The structure-preserving isomorphism interchanges $0$ and $1$.

\section{An evaluation identity used in the proofs}
\begin{lemma}\label{lem:beta}
Let $u:X\times A\to B$ have transpose $\Lambda u:X\to B^A$.
For $p:Z\to X$ and $q:Z\to A$,
\[
 \ev\circ\pair{(\Lambda u)\circ p}{q}
 =u\circ\pair{p}{q}.
\]
\end{lemma}
\begin{proof}
Precompose
$\ev\circ((\Lambda u)\times\id_A)=u$
with $\pair{p}{q}:Z\to X\times A$ and use the two
projection equations for the product.
\end{proof}
This identity accounts for the evaluation steps in the CCC
proof and in Proposition~\ref{prop:uniform}. It is preferable
to an unspecified appeal to ``naturality'' when checking
which morphisms are being composed.

\section*{Online primary sources}\label{sec:online-sources}
\addcontentsline{toc}{section}{Online primary sources}
The following sources complement the DOI-identified research literature.
They are teaching material or a commented formal development; the worked
calculations in this survey are given explicitly in the corresponding examples.
\begin{enumerate}[label=O\arabic*,ref=O\arabic*,leftmargin=*]
\item\label{online:horwitz} University of Wisconsin--Madison,
\emph{CS704 notes: Domain Theory}. Discussion of lifted domains,
recursive definitions, and least fixed points.
\url{https://pages.cs.wisc.edu/~horwitz/CS704-NOTES/8.DOMAIN-THEORY.html}.
Accessed 12 October 2026.
\item\label{online:waterloo} University of Waterloo,
\emph{CS360 course notes, Lecture 19: Recursion Theorem and Rice's Theorem}
(2021). Quines and semantic undecidability.
\url{https://cs.uwaterloo.ca/~s4bendav/files/CS360S21Lec19.pdf}.
Accessed 12 October 2026.
\item\label{online:escardo} Mart\'in Escard\'o,
\emph{Various.LawvereFPT}, TypeTopology, commented Agda development
(initially dated 5 September 2018).
\url{https://martinescardo.github.io/TypeTopology/LawvereFPT.html}.
Accessed 12 October 2026.
\end{enumerate}

\bibliographystyle{plain}
\bibliography{references}
\end{document}